\documentclass[a4paper,12pt]{article}
\usepackage{glaudoenpkg}
\usepackage{setspace}

\makeatletter

\renewcommand{\theresults}{\ifnum0=\c@subsection\thesection.\arabic{results}\else\thesubsection.\arabic{results}\fi}

\makeatother

\doublespace

\DeclareMathOperator{\Fix}{\mathfrak{Fix}}
\DeclareMathOperator{\Dom}{\mathfrak{Dom}}

\newcommand\ack{\section*{Acknowledgements}}

\begin{document}

	\title{Elliptic fibrations and Hilbert Property}
	\author{Julian Lawrence Demeio}

\maketitle

\begin{abstract}
	For a number field $K$, an algebraic variety $X/K$ is said to have the Hilbert Property if $X(K)$ is not thin. We are going to describe some examples of algebraic varieties, for which the Hilbert Property is a new result.
	
	The first class of examples is that of smooth cubic hypersurfaces with a $K$-rational point in $\P_n/K$, for $n \geq 3$. These fall in the class of unirational varieties, for which the Hilbert Property was conjectured by Colliot-Thélène and Sansuc.
	
	We then provide a sufficient condition for which a surface endowed with multiple elliptic fibrations has the Hilbert Property.  As an application, we prove the Hilbert Property for a class of K3 surfaces, and some Kummer surfaces.
\end{abstract}

\section{Introduction}

This paper is concerned with providing some examples of varieties with the \textit{Hilbert Property}, concerning the set of $k$-rational points $X(k)$, for an algebraic variety $X$ defined over a field $k$.

A geometrically irreducible variety $X$ over a field $k$ is said to have the \textit{Hilbert Property} if, for any finite morphism $\pi:E \rightarrow X$, such that $X(k)\setminus \pi(E(k))$ is not Zariski-dense in $X$, there exists a rational section of $\pi$ (see  \cite[Ch. 3]{Serre} for an introduction of the Hilbert Property).

Some motivation for the study of the Hilbert Property comes from the following conjecture, of which a proof would settle the Inverse Galois Problem, as noted in \cite{congetturastrong}:

\begin{conjecture}[Colliot-Thélène, Sansuc]\label{congetturona}
	Let $X$ be a unirational variety over a number field, then $X$ has the Hilbert Property.
\end{conjecture}

The first result of this paper, in Section \ref{cubics}, is the Hilbert Property for smooth cubic hypersurfaces of $\dim \geq 2$, defined over a number field $K$ and with at least one $K$-rational point. Since, under these hypothesis, cubic hypersurfaces are $K$-unirational \cite{segre}, this result gives positive examples of Conjecture \ref{congetturona}. In Section \ref{surfaces} we prove the following theorem, which generalizes \cite[Theorem 1.4]{myarticle1}.

\begin{theorem}\label{Thm:teoremone}
	Let $E$ be a projective smooth geometrically connected algebraic surface, defined over a number field $K$. Suppose that there exist $n \geq 2 $ elliptic fibrations $\pi_1,\dots,\pi_n:E \rightarrow \P_1$ . Let $F \subset E$ denote the union of the divisors of $E$ that are contained, for each $i=1,\dots, n$, in a fiber of $\pi_i$. If $E \setminus F$ is simply connected and $E(K)$ is Zariski-dense in $E$, then $E$ has the Hilbert Property.
\end{theorem}

We then use this theorem to give explicit examples of K3 surfaces with the Hilbert Property.
The examples are produced starting from a construction presented in \cite{garbagnati}, by Garbagnati and Salgado. Finally, we prove the Hilbert Property for some Kummer surfaces, for which the Hilbert Property was suggested to be true by Corvaja and Zannier \cite{articoloHP}.

\section{Background}

This section contains some preliminaries. In particular, in the last paragraph, we shall recall some known results, that will be used in later sections, concerning the Hilbert Property. Moreover, we shall take care here of most of the notation that will be used in the paper.

\paragraph{Notation}

Throughout this paper, except when stated otherwise, $k$ denotes a perfect field and $K$ a number field. A ($k$-)\textit{variety} is an algebraic quasi-projective variety (defined over a field $k$), not necessarily irreducible or reduced. Unless specified otherwise, we will always work with the Zariski topology. 

Given a morphism $f:X \rightarrow Y$ between $k$-varieties, and a point $s:\Spec(k(s))\rightarrow Y$, we denote by $f^{-1}(s)$ the scheme-theoretic fibered product $\Spec(k(s))\times_Y X$, and call it the \textit{fiber} of $f$ in $s$. Hence, with our notation, this is not necessarily reduced.

A geometrically integral $k$-variety $X$ is a $k$-variety such that $X_{\bar{k}}$ is integral. 

A proper morphism $f:Y \rightarrow X$ between normal $k$-varieties is a \textit{cover} if the fiber $f^{-1}(\eta)$ is finite for every point $\eta$ of codimension $\leq 1$ in $X$.

When $k \subset \C$, a smooth $k$-variety $X$ is \textit{simply connected} if $X_{\C}$ is a simply connected topological space. 

Given a morphism $f:Y \rightarrow X$ between integral $k$-varieties, with $Y$ normal, we will make use of the notion of \textit{relative normalization} of $X$ in $Y$. We refer the reader to \cite[Def. 28.50.3, Tag 0BAK]{stacks-project} or \cite[Def. 4.1.24]{Liu} for its definition, and recall here the properties that are needed in this paper. Namely, we will need that the \textit{relative normalization} of $X$ in $Y$ is a finite morphism $n:\hat{X}\rightarrow X$ such that $\hat{X}$ is normal, and such that there exists a factorization $f=\phi \circ n$, where $\phi:Y \rightarrow \hat{X}$ has a geometrically integral generic fiber. The \textit{normalization} of an integral variety $X$ is the usual normalization $\hat{X}$ of $X$ \cite[Section 4.1.2]{Liu}. 

The domain $\Dom(f)$ of a rational morphism $f:Y \dashrightarrow X$ between integral $k$-varieties is the maximal open Zariski subset $U \subset Y$ such that $\restricts{f}{U}$ extends to a morphism on $U$.

Given a rational morphism $f:Y \dashrightarrow X$ between integral $k$-varieties, and a birational transformation $b: Y' \dashrightarrow Y$, we will denote with abuse of notation, when there is no risk of confusion, the morphism $f \circ b$ by $f$. We say that $f$ is \textit{well-defined} on $Y'$ if $\Dom(f \circ b)=Y'$. 

\paragraph{Ramification}

We recall that a morphism $f:Y \rightarrow X$ between $k$-varieties is \textit{unramified} (resp. \textit{étale}) in $y \in Y$ if its differential $df_y:T_yY\rightarrow T_{f(y)}X$ is injective (resp. an isomorphism). Otherwise we say that $f$ is \textit{ramified} at $y$. The set of points where $f$ is ramified has a closed subscheme structure in $Y$, and we will refer to it as the \textit{ramification locus}. The image of the ramification locus under $f$ is the \textit{diramation locus}. We recall that, by Zariski's Purity Theorem \cite[Lem. 53.20.4, Tag 0BMB]{stacks-project}, when $Y$ is normal and $X$ is smooth, the diramation locus of a finite morphism $f:Y \rightarrow X$ is a divisor. Hence, in this case, we will also refer to the diramation locus as the \textit{diramation divisor}. 
A simply connected variety $X$ does not have any geometrically integral cover of degree $>1$ which is unramified in codimension $1$.

\paragraph{Cubic hypersurfaces}

Cubic hypersurfaces are hypersurfaces in $\P_n$ defined by a cubic homogeneous polynomial. We recall the following:

\begin{theorem}[Segre '43, Manin '72]\label{Thm:unirationalcubics}
	Let $X$ be a smooth cubic hypersurface, defined over a field $k$, with a $k$-rational point. Then $X$ is unirational.
\end{theorem}

\paragraph{Hilbert Property}

For a more detailed exposition of the basic theory of the Hilbert Property and thin sets we refer the interested reader to \cite[Ch. 3]{Serre}. We limit ourselves here to recalling the most common definition (which is not the one given in the introduction), and some recent results.

\begin{definition}\label{Def:thin}
	Let $X$ be a geometrically integral variety, defined over a field $k$. A \textit{thin} subset $S \subset X(k)$ is any set contained in a union $D(k)\cup \bigcup_{i=1,\dots,n} \pi_i(E_i(k))$, where $D \subsetneq X$ is a subvariety, and $\pi_i:E_i \rightarrow X$ are generically finite morphisms\footnote{When $X$ is normal, one can substitute here ``generically finite morphisms" with ``covers" to get an equivalent definition.} of degree $>1$, and the $E_i$'s are irreducible.
\end{definition}

\begin{remark}\label{HPthin}
	A geometrically integral $k$-variety $X$ has the Hilbert Property if and only if $X(k)$ is not thin.
\end{remark}

Throughout this paper, we will use the abbreviation HP for Hilbert Property.

\begin{theorem}[Bary-Soroker, Fehm, Petersen]\label{fibrationHP}
	Let $f:X \rightarrow S$ be a morphism of $K$-varieties. Suppose that there exists a non-thin subset $A \subset S(K)$ such that, for each $s \in A$, $f^{-1}(s)$ has the HP. Then $X/K$ has the HP.
\end{theorem}
\begin{proof}
	See \cite[Theorem 1.1]{fibrazioniHP}.
\end{proof}

\begin{definition}\label{ellipticfibration}
	Let $\Eps$ be a normal projective algebraic $K$-surface. We say that a morphism $\pi:\Eps \rightarrow \P_1$ is an \textit{elliptic fibration} if its generic fiber is a smooth, geometrically connected, genus $1$ curve.
\end{definition}

The following theorem is included just for completeness. In fact, Theorem \ref{Thm:teoremone}, which we are going to prove in Section \ref{surfaces}, is a stronger version of it. 

\begin{theorem}\label{mydoublefibration}
	Let $K$ be a number field, and $E$ be a projective smooth simply connected algebraic $K$-surface, endowed with two elliptic fibrations $\pi_i:E \rightarrow \P_1/K, \ i=1,2$, such that $\pi_1\times \pi_2:E \rightarrow \P_1 \times \P_1$ is a finite morphism. Suppose moreover that the following hold:
	\begin{itemize}
		\item[(a)] The $K$-rational points $E(K)$ are Zariski-dense in $E$;
		\item[(b)] Let $\eta_1 \cong \Spec K(\lambda)$ be the generic point of the codomain of $\pi_1$. All the diramation points (i.e. the images of the ramification points) of the morphism $\restricts{\pi_2}{\pi_1^{-1}(\eta_1)}$ are non-constant in $\lambda$, and the same holds upon inverting $\pi_1$ and $\pi_2$.
	\end{itemize}
	Then the surface $E/K$ has the Hilbert Property.
\end{theorem}

\begin{proof}
		See \cite[Theorem 1.4]{myarticle1}.
\end{proof}

\section{Hilbert Property for cubic hypersurfaces}\label{cubics}

The case $n=2$ of the following theorem is a consequence of a result of Swinnerton-Dyer \cite{cubiche-swinndyer}. However, we still present an independent proof here, as it is particularly simpler than Swinnerton-Dyer's proof.

\begin{theorem}\label{cubiche}
	Let $X\subset \P_n/K$, $n \geq 3$ be a smooth cubic hypersurface, with a $K$-rational point. Then $X$ has the Hilbert Property.
\end{theorem}

In this section our base field will always be a number field $K$.

We need the following lemma, of which an explicitly computable version can be found in \cite{RonaldvanLuijk2012}.

\begin{lemma}\label{positivegenericrank}
	Let $\pi:\Eps \rightarrow \P_1$ be an elliptic fibration, defined over a number field $K$. Then, there exists an open Zariski subset $U_{\pi} \subset \Eps$ such that, for any $P \in U_{\pi}(K)$, $\pi^{-1}(\pi(P))$ is smooth and $\#{\pi^{-1}(\pi(P))(K)}=\infty$. 
\end{lemma}

\begin{proof}[Proof of Theorem \ref{cubiche}]
	
	We note that $X$ is $K$-unirational by Theorem \ref{Thm:unirationalcubics}, in particular it has Zariski-dense $K$-rational points.
	
	We prove the result by induction on $n$.
	
	\emph{Case} $n=3$.

	We assume by contradiction that $X$ does not have the HP. Then there exist irreducible covers $\phi_i: Y_i \rightarrow X, \ i=1, \dots, m$ of degree $\deg \phi_i >1$ and a divisor $D \subset X$ such that $X(K) \subset \cup_i \phi_i(Y_i(K)) \cup D(K)$. We may assume, without loss of generality, that the $Y_i$'s are normal and geometrically integral\footnote{In fact, possibly by enlarging $D$, one can substitute $Y_i$ with the normalization of $X$ in $Y_i$. A normal variety that is not geometrically integral over the base field $K$ does not have any $K$-rational points.}, and that the $\phi_i$'s are finite morphisms. Let us denote now by $R_i$ the diramation divisor of $\phi_i$. By Lefschetz' hyperplane Theorem, $X$ is simply connected, hence we know that the $R_i$'s are nonempty for each $i=1, \dots, m$. 
	
	Let us denote by $\A_4^*$ the dual affine space of $\A_4$, minus the origin. To each element $h \in \A_4^*$ corresponds a hyperplane $H(h)\subset \P_3$ in a canonical way\footnote{Namely, if $\lambda \in \A_4^*$, the associated hyperplane is $\{\mathbf{x} \in \P_3 \ | \ \lambda(\mathbf{x})=0\}$.}. Let $(h_1,h_2) \in \A_4^*(K) \times \A_4^*(K)$ be such that 
	\begin{enumerate}
		\item $H(h_1) \cap H(h_2) \cap X$ is (a scheme consisting of) three distinct points (and hence, as a direct consequence, $H(h_1) \cap H(h_2)$ is a line), and it is disjoint from the union of the $R_i$'s;
		\item $H(h_1) \cap X$ and $H(h_2) \cap X$ are smooth curves;
		\item The morphism $[h_1:h_2]:X \setminus H(h_1) \cap H(h_2) \rightarrow \P_1$ is non-constant on each of the irreducible components of the $R_i$'s.
	\end{enumerate}
	We note that, since all conditions are open and non-empty, such a couple $(h_1,h_2)$ always exists.

	Let $\{P_1,P_2,P_3\}$ be the intersection $H(h_1) \cap H(h_2) \cap X$, and let $\pi:  X \setminus \{P_1,P_2,P_3\}  \rightarrow \P_1$ be the following morphism:
	\begin{equation*}
		P \longmapsto [h_1(P):h_2(P)].
	\end{equation*}

	The map $\pi$ extends naturally to a morphism $\hat{\pi}:\hat{X} \rightarrow \P_1$, where $\hat{X}=\Blowup_{P_1+P_2+P_3}X$ denotes the blowup of $X$ in the (smooth) subscheme $P_1+P_2+P_3 \subset X$. We note that, since $X$ is a cubic surface, the morphism $\hat{\pi}$ is an elliptic fibration.

	We claim now that the morphisms $\pi \circ \phi_i:Y_i \setminus \phi_i^{-1}(\{P_1,P_2,P_3\}) \rightarrow \P_1$ have geometrically integral generic fiber for each $i=1, \dots, m$. In fact, let us assume by contradiction that there existed an $i \in \{1, \dots, m\}$ such that this is not true.
	Let
	\begin{equation}\label{Eq:factorization1}
	\pi \circ \phi_i : Y_i \setminus \phi_i^{-1}(\{P_1,P_2,P_3\}) \xrightarrow{\pi'} C \xrightarrow{r} \P_1
	\end{equation}
	be the relative normalization factorization of $\restricts{\pi \circ \phi_i}{Y_i \setminus \phi_i^{-1}(\{P_1,P_2,P_3\})}$. We would have that $\deg r>1$.
	The factorization \ref{Eq:factorization1} yields a morphism $\phi'_i:Y_i \setminus \phi_i^{-1}(\{P_1,P_2,P_3\}) \rightarrow X \setminus \{P_1,P_2,P_3\}\times_{\P_1}C$, and a factorization:
	\[
	\phi_i: Y_i \setminus \phi_i^{-1}(\{P_1,P_2,P_3\}) \xrightarrow{\phi'_i} \widehat{X \setminus \{P_1,P_2,P_3\}\times_{\P_1}C} \xrightarrow{\alpha}  X \setminus \{P_1,P_2,P_3\},
	\] 
	where $\widehat{X \setminus \{P_1,P_2,P_3\}\times_{\P_1}C}$ denotes the normalization of $X \setminus \{P_1,P_2,P_3\}\times_{\P_1}C$.
	
	Hence the diramation of $\phi_i$ would contain the diramation of $\alpha$, which would be nonempty if $\deg r>1$ (since $X \setminus \{P_1,P_2,P_3\}$ is simply connected) and contained in a finite union of fibers of $\pi$. This contradicts our choice of $(h_1,h_2)$.

	Let us denote now by $\hat{Y_i}$ the desingularization of $Y'_i=Y_i \times_X \hat{X}$, and by $\psi_i:\hat{Y_i}\rightarrow \P_1$  the composition of the desingularization morphism $\hat{Y_i} \rightarrow Y'_i$, the projection $Y'_i \rightarrow \hat{X}$ and the map $\hat{\pi}:\hat{X} \rightarrow \P_1$. By the Theorem of generic smoothness \cite[Corollary 10.7]{hartshorne} we know that there exists an open subset $V_i \subset \P_1$ such that, for each $t \in V_i(\overline{K})$, $\psi_i^{-1}(t)$ is smooth (and we may assume, by further restricting $V_i$, geometrically connected as well, because $\psi_i$ has a geometrically integral generic fiber). Let us denote now by $D'\subset \hat{X}$ the proper subvariety, which is the union of all the following:
	\begin{enumerate}
		\item the fibers $\hat{\pi}^{-1}(x)$, for each $x \notin V_i$, for each $i=1, \dots, m$;
		\item the proper transform of $D \subset X$, and the exceptional locus of $\hat{X}\rightarrow X$;
		\item the proper transform of $X \setminus U$, where $U$ is defined as in Lemma \ref{positivegenericrank} for $(\Eps,\pi)=(\hat{X},\hat{\pi})$. 
	\end{enumerate}
	Let us choose now a $K$-rational point $P \in (\hat{X}\setminus D')(K)$, and let us denote by $E_P$ the fiber $\hat{\pi}^{-1}(\hat{\pi}(P))$. We know, by Lemma \ref{positivegenericrank}, that $E_P$ has infinitely many $K$-rational points. We have assumed, however, that  $X(K) \subset \cup_i \phi_i(Y_i(K)) \cup D(K)$, 
	 and hence 
	\begin{equation}\label{curvaeP}
	E_P(K) \subset \cup_i \psi_i^{-1}(\hat{\pi}(P))(K) \cup (E_P\cap D')(K). 
	\end{equation}

	We claim that the right hand side of \ref{curvaeP} is finite. In fact, for each $i=1,\dots, m $, the morphism $\psi_i^{-1}(\hat{\pi}(P)) \rightarrow E_P$ is ramified by the invariance of the ramification locus under base change (see e.g. \cite[Tag 0C3H]{stacks-project}), and, since the curve $\psi_i^{-1}(\hat{\pi}(P))$ is a smooth geometrically connected complete curve, by Riemann-Hurwitz theorem, it is of genus $>1$. As a consequence, by Faltings' theorem, $\psi_i^{-1}(\hat{\pi}(P))(K)$ is finite for each $i=1,\dots, m$. Moreover, $(E_P\cap D')$ is obviously finite, hence we have proved that the right hand side of \ref{curvaeP} is finite. As we noted before, however, $E_P(K)$ is infinite. We have obtained a contradiction, proving the theorem in the case $n=3$.

	\emph{Case} $n \geq 4$.

	By Bertini's theorem \cite[Remark 10.9.2]{hartshorne}, we know that there exists a Zariski-open subset $U \subset X$ such that, for each $P \in U(\bar{K})$, the generic hyperplane of $\P_n$ passing through $P$ cuts $X$ in a smooth irreducible cubic of dimension $n-2$.
	
	 We choose now a $K$-rational point $P \in U(K)$ 
	, and two $K$-rational (distinct) hyperplanes $H_0\defeq\{h_0=0\}, H_{\infty}\defeq \{h_{\infty}=0\}$ passing through $P$, such that $H_0 \cap X$ is smooth. Let $L \defeq H_0 \cap H_{\infty}$.
	
	Let us consider now the following morphism:
	\begin{equation*}
		\phi:  X \setminus L \cap X \longrightarrow \P_1, \quad
		P \longmapsto [h_0(P):h_{\infty}(P)],
	\end{equation*}
	which extends naturally to a morphism $\hat{\phi}: \Blowup_{L \cap X}X \rightarrow \P_1$. For $t=[t_1:t_2] \in \P_1$, the scheme-theoretic fiber $\hat{\phi}^{-1}(t)$ is isomorphic to the intersection $H_t \cap X$, where $H_t$ denotes the hyperplane $t_1H_0+t_2H_{\infty}=0$ in $\P_n$. Since $H_0 \cap X$ is smooth, the intersection $H_t \cap X$ is smooth for $t$ in a Zariski open subset $V \subset \P_1(\bar{K})$, containing $0$. For $x \in V(K)$, the fiber $\hat{\phi}^{-1}(x)$ is a smooth cubic in an $(n-1)$-dimensional projective space, with a $K$-rational point in it (namely, $P$). Hence, by the induction hypothesis, this fiber has the HP, and hence, since $\P_1/K$ has the HP, $\Blowup_{L \cap X}X$ (and, therefore, $X$) has the HP as well by Theorem \ref{fibrationHP}.
\end{proof}

\section{Surfaces with the Hilbert Property}\label{surfaces}

The proof of Theorem \ref{Thm:teoremone} uses the following lemma, which is Lemma 3.2 of \cite{articoloHP}.

\begin{lemma}\label{lemmagruppi}
	Let $G$ be a finitely generated abelian group of positive rank. Let $n \in \N$ and $\{h_u+H_u\}_{u=1,\dots,n}$ be a collection of finite index cosets in $G$, i.e. $h_u \in G, \ H_u < G$ and $[G:H_u]<\infty$ for each $u=1,\dots, n$.
	If $G \setminus \bigcup_{u=1,\dots,n} (h_u+H_u)$ is finite, then $\bigcup_{u=1,\dots,n} (h_u+H_u)=G$.
\end{lemma}

\begin{notation}
	Let $E$ be a smooth projective geometrically connected $k$-surface, endowed with fibrations $\pi_1, \dots, \pi_n:E \rightarrow \P_1, \ n \geq 2$. We call the \textit{fixed locus of} $\pi_1,\dots,\pi_n$ the following reduced subvariety of $E$:
	\[
	\Fix(\pi_1,\dots,\pi_n)= \bigcup \{D: D \text{ is a divisor in } E \text{ and } \restricts{\pi_i}{D} \text{ is constant } \forall i=1,\dots,n \}
	\]
\end{notation}

\begin{remark}
	The subvariety $F \subset E$ described in Theorem \ref{Thm:teoremone} is exactly $\Fix(\pi_1,\dots,\pi_n)$.
\end{remark}

\begin{proof}[Proof of Theorem \ref{Thm:teoremone}]
	Suppose by contradiction that there exist $m \in \N$, irreducible covers $\phi_i:Y_i \rightarrow E, \ i=1,\dots, m$ and a proper subvariety $D \subsetneq E$, such that $E(K) \subset D(K) \cup \cup_i \phi_i(Y_i(K))$, and $\deg \phi_i \geq 2$. We may assume, without loss of generality, that the $Y_i$'s are smooth and geometrically connected.
	
	We say that a cover $Y_i$, and the corresponding $i$, is $\{j_1,\dots,j_k\}$-\textit{unramified}, where $\{j_1,\dots,j_k\}\subset\{1,\dots, m\}$, if, for each $j \in \{j_1,\dots,j_k\}$, the diramation divisor of $\phi_i$ is contained in a finite union of fibers of $\pi_j$. We say that it is $\{j_1,\dots,j_k\}$-\textit{ramified} otherwise. By the simply connectedness hypothesis, no cover is $\{1,\dots,m\}$-unramified.
	
	We say that a point $P \in E(K)$ is $j$-\textit{good} if the fiber $\pi_j^{-1}(\pi_j(P))$ is smooth and geometrically connected of genus $1$, and $\#\pi_j^{-1}(\pi_j(P))(K)= \infty$. By Lemma \ref{positivegenericrank}, there exists an open subset $U\subset E$ such that, for each $P \in U(K)$, $P$ is $j$-good for each $j=1,\dots,m$. We assume moreover, without loss of generality, that $U \cap D= \emptyset$.
	
	For each $1 \leq i \leq m, \ 1\leq j \leq n$, let
	\[
	\pi_j \circ \phi_i:Y_i \xrightarrow{r_{ij}} C_{ij} \rightarrow \P_1
	\]
	be the relative normalization factorization of $\pi_j \circ \phi_i$. We have that the geometric generic fiber of $r_{ij}$ is irreducible, $C_{ij}$ is a smooth complete geometrically connected curve and $C_{ij} \rightarrow \P_1$ is a finite morphism. Let $\widehat{E \times_{\P_1} C_{ij}} \rightarrow E \times_{\P_1} C_{ij}$ be a desingularization of $E \times_{\P_1} C_{ij}$. Then there exists a commutative diagram as follows:
	\begin{center}
		\begin{tikzcd}
			\hat{Y_i} \arrow[r, "\psi_{ij}"] \arrow[d, "b_i"] & \widehat{E \times_{\P_1} C_{ij}} \arrow[r] & E \\
			Y_i \arrow[rru, swap, "\psi_i"] \\
		\end{tikzcd}
	\end{center}
	where $b_i: \hat{Y_i} \rightarrow Y_i$ is the composition of a finite sequence of blowups.

	When $i$ is $j$-ramified, since the diramation locus of $\psi_i$ contains at least one component transverse to the fibration $\pi_j$, and the morphism $\widehat{E \times_{\P_1} C_{ij}} \rightarrow E$ is $j$-unramified, $\psi_{ij}$ must have at least one irreducible component of the diramation locus which is transverse to the fibration $\widehat{E \times_{\P_1} C_{ij}} \rightarrow C_{ij}$.
	Hence, by the invariance of the ramification locus under base change \cite[Tag 0C3H, Lemma 30.10.1]{stacks-project}, when $i$ is not $j$-unramified, the geometric generic fiber of $\hat{Y_i}\rightarrow C_{ij}$, which is isomorphic to the geometric generic fiber of $r_{ij}$, is ramified over the geometric generic fiber of $\widehat{E \times_{\P_1} C_{ij}} \rightarrow C_{ij}$, which has genus $1$. Therefore, it is a curve of genus $>1$. Hence, when $i$ is $j$-ramified, the geometric generic fiber of $\pi_j \circ \phi_i$ is a finite union of curves of genus $>1$.  When $i$ is $j$-unramified one shows analogously that the geometric generic fiber of $\pi_j \circ \phi_i$ is a finite union of curves of genus $1$.
	
	We have that, for each $P \in U(K)$:
	\[
	P \in \pi_n^{-1}(\pi_n(P))(K)\cap U\subset\]\[ \bigcup_{i \ n-\text{unramified}} \phi_i(Y_i(K))\cap \pi_n^{-1}(\pi_n(P)) \cup \bigcup_{i \  n-\text{ramified}} \phi_i(Y_i(K))\cap \pi_n^{-1}(\pi_n(P))=
	\]	
	\[
	\bigcup_{i \ n-\text{unramified}} \phi_i((\pi_n\circ \phi_i)^{-1}(\pi_n(P)(K))) \cup \bigcup_{i \  n-\text{ramified}} \phi_i((\pi_n\circ \phi_i)^{-1}(\pi_n(P)(K))).
	\]
	Hence, as noted before, when $i$ is $n$-ramified, after restricting (without loss of generality) $U$ to a smaller Zariski open subset, $(\pi_n\circ \phi_i)^{-1}(\pi_n(P))$ is a curve of genus $>1$. Therefore, by Falting's Theorem, we deduce that:
	\[
	\pi_n^{-1}(\pi_n(P))(K) \subset \bigcup_{i \ n-\text{unramified}} \phi_i((\pi_n\circ \phi_i)^{-1}(\pi_n(P)(K))) \cup A_0(P),
	\]
	where $A_0(P)$ is a finite set. Moreover, when $i$ is $n$-unramified, after restricting (without loss of generality) $U$ to a smaller Zariski open subset, $(\pi_n\circ \phi_i)^{-1}(\pi_n(P))$ is a curve of genus $1$. Therefore, by the weak Mordell-Weil theorem, we have that, for each $i=1,\dots, m$, $(\pi_n\circ \phi_i)^{-1}(\pi_n(P))(K)\subset \pi_n^{-1}(\pi_n(P))(K)$ is either empty or a finite index coset.
	
	Hence, by Lemma \ref{lemmagruppi}, we deduce that:
	\[
	P \in \pi_n^{-1}(\pi_n(P))(K) \subset \bigcup_{i \ n-\text{unramified}} \phi_i (Y_i(K)).
	\]
	Therefore we have that
	\[
	E(K)\subset \bigcup_{i \ n-\text{unramified}} \phi_i (Y_i(K)) \cup (E\setminus U).
	\]
	We have now reduced to the case where all the covers $Y_i$ are $n$-unramified. Proceeding with an easy induction on $n$, we may reduce to the case where all the covers $Y_i$ are $\{1,\dots,n\}$-unramified. But there are no such covers by hypothesis, whence we deduce that $E(K)$ is not Zariski-dense in $E$, leading to the desired contradiction.
\end{proof}

\subsection{A family of K3 surfaces with the Hilbert Property}

As an application of Theorem \ref{Thm:teoremone} we describe now a family of K3 surfaces\footnote{K3 surfaces (and, in general, Calabi-Yau varieties) represent a ``limiting case" for the study of rational points, at least conjecturally. In fact, the conjectures of Vojta suggest that on algebraic varieties there should be ``less" rational points as the canonical bundle gets ``bigger". Hence, since for K3 surfaces the canonical bundle is trivial by definition, we expect the rational points here not to be ``too much", yet their existence (and Zariski-density) is not precluded. In fact, proving the HP, we are providing some examples of abudance of rational points in such surfaces.} with the Hilbert Property.

For $\lambda \in K^*$, and $c_1,c_2 \in K[x,y,z]$ cubic homogeneous polynomials, let 
\begin{equation}\label{kummer}
X'_{\lambda}(c_1,c_2)\defeq\{([w_0:w_1],[x:y:z])\in \P_1\times \P_2\ | \ w_0^2c_1(x,y,z)=\lambda w_1^2c_2(x,y,z)\}.
\end{equation}

The surfaces $X'_{\lambda}(c_1,c_2)$ are, up to a birational transformation, endowed with multiple elliptic fibrations, usually defined over $\bar{K}$, whose construction we recall in the next paragraphs.
In some particular cases, when enough of these fibrations are defined over $K$, this allows us to use Theorem \ref{Thm:teoremone} to prove the Hilbert Property for these varieties.

 \smallskip

\begin{remark}\label{Rmk:casoparticolarekummer}
	When $c_1(x,y,z)=f_1(x,z)$ does not depend on $y$, $c_2(x,y,z)=f_2(y,z)$ does not depend on $x$, and both $f_1$ and $f_2$ do not have multiple roots, equation (\ref{kummer}) describes a Kummer surface (i.e. a quotient of an abelian surface by the group of isomorphisms $\{\pm 1\}$).
	
	In fact, in this case, equation (\ref{kummer}) describes, up to a birational transformation, the quotient of $E_1 \times E_2$ by the group $\{\pm 1\}$, where $E_1$ and $E_2$ are the elliptic curves defined by the following Weierstrass equations:
	\begin{equation}\label{ellittichechecopronolakummer}
	E_1: \ w^2=f_1(x,z), \quad
	E_2: \ w^2=f_2(y,z).
	\end{equation} 
\end{remark}

\paragraph{Construction of the Elliptic Fibrations}

 We give now an explicit construction of a smooth model of of $X'_{\lambda}(c_1,c_2)$ and of the elliptic fibrations it is endowed with, under a genericity assumption on $c_1$ and $c_2$. We avoid going into detail, as these constructions are described thoroughly by Garbagnati and Salgado in \cite{garbagnati}.

Let $P_1,\dots ,P_9$ be $9$ (distinct) points in $\P_2(\bar{K})$ such that:
\begin{enumerate}
	\item $P_1,\dots, P_4$ are the four points of intersection of two smooth conics $Q^1_1\defeq \{q^1_1=0\},Q^1_2\defeq \{q^1_2=0\}$ in $\P_2$, defined over $K$;
	\item $P_5,\dots, P_8$ are the four points of intersection of two smooth conics $Q^2_1\defeq \{q^2_1=0\},Q^2_2\defeq \{q^2_2=0\}$ in $\P_2$, defined over $K$;
	\item The eight points $P_1,\dots, P_8$ are in generic position\footnote{\label{genericpoints}By generic position, we mean that no three of these points lie on a line, and no six of these points lie on a conic.};
	\item $P_1, \dots, P_9$ are the nine points of intersection of two smooth cubics $C_1\defeq \{c_1=0\},C_2\defeq \{c_2=0\}$ in $\P_2$, defined over $K$.
\end{enumerate}

\begin{definition}\label{good}
	We say that a $9$-tuple $(P_1,\dots, P_9) \in P_2(\bar{K})^9$ is \emph{good} if it satisfies the four conditions above.
\end{definition}

We note that, by choosing two sufficiently Zariski generic $4$-tuples $p_1, \dots, p_4$ and $p_5, \dots, p_8$ such that both $p_1+\dots +p_4$ and $p_5+\dots+p_8$ are defined over $K$, and letting $p_9$ be the unique ninth intersection of the pencil of cubics through $p_1,\dots,p_8$, the $9$-tuple $p_1,\dots,p_9$ is good.

We assume hereafter that a choice of a good $9$-tuple of points $P_1,\dots,P_9$, of the two cubics $C_1, C_2$ and of the conics $Q^i_j, \ i,j=1,2$ has been made.

Let $R\defeq \Blowup_{P_1+\dots+P_9}\P_2$ be the blowup of $\P_2$ in the nine points $P_1,\dots,P_9$. The two cubics $C_1,C_2$ define an elliptic fibration on $R$, which we denote by $\mathcal{C}$, defined as $\mathcal{C}(p)=[c_1(p):c_2(p)]$.

The fibers of $\mathcal{C}$ are by construction the proper transforms of the elements of the pencil generated by $C_1, C_2$. 

For $\lambda \in K^*$, let $f_{\lambda}:\P_1 \rightarrow \P_1$ be the morphism defined by $f_{\lambda}([w_0:w_1])=[w_0^2:\lambda w_1^2]$. Let also $X_{\lambda}(c_1,c_2)$ be the smooth surface defined as the fibered product $R \times_{\mathcal{C},f_{\lambda}} \P_1$, $\alpha_{\lambda}:X_{\lambda}(c_1,c_2)\rightarrow R$ be the projection on the first factor, and $\phi_{\lambda}:X_{\lambda}\rightarrow \P_1$ be the projection on the second factor. The surface $X_{\lambda}(c_1,c_2)$ is a K3 surface.

\begin{note}
	We observe that, by construction, $X_{\lambda}(c_1,c_2)$ is birational to $X'_{\lambda}(c_1,c_2)$.
\end{note}

The surface $X_{\lambda}(c_1,c_2)$ is endowed with at least three elliptic fibrations. The first one is $\phi_{\lambda}$. The second and third one, which we denote by $\widetilde{\mathcal{Q}}^1$ and $\widetilde{\mathcal{Q}}^2$, are the proper transforms of the two pencils of conics generated, respectively, by $\{Q^1_1,Q^1_2\}$ and by $\{Q^2_1,Q^2_2\}$. I.e., $\widetilde{\mathcal{Q}}^i=\mathcal{Q}^i \circ \alpha_{\lambda}$, where the maps $\mathcal{Q}^i:R \rightarrow \P_1$  are defined as $\mathcal{Q}^i(p)=[q^i_1(p):q^i_2(p)], \ i=1,2$.

\begin{proposition}\label{Thm:Zariskidense}
	Let $P_1,\dots,P_9 \in \P_2(\bar{K})$ be a good $9$-tuple of points, and $C_1\defeq\{c_1=0\}, C_2\defeq\{c_2=0\}$ be two smooth cubics such that $C_1\cap C_2=\{P_1,\dots,P_9\}$. If $X_{\lambda}(c_1,c_2)$ has Zariski-dense $K$-rational points, then it has the Hilbert Property.
\end{proposition}

\begin{proof}
	Since $\Fix(\phi_{\lambda}, \widetilde{\mathcal{Q}}^1, \widetilde{\mathcal{Q}}^2)= \emptyset$, this is an immediate consequence of Theorem \ref{Thm:teoremone} applied to $X_{\lambda}(c_1,c_2)$ (which, being a K3 surface, is simply connected), with fibrations $\phi_{\lambda}, \widetilde{\mathcal{Q}}^1$ and $\widetilde{\mathcal{Q}}^2$.
\end{proof}

\begin{remark}\label{Rmk:riassunto}
	In general, looking at the explicit equations, it is easy to show that, given $c_1, c_2$ cubic polynomials, there exist always infinitely many $\lambda \in K^*$ such that $X_{\lambda}(c_1,c_2)$ has Zariski-dense $K$-rational points. Hence, as a corollary of Proposition \ref{Thm:Zariskidense}, one obtains that, for any $c_1,c_2$, there exist always infinitely many $X_{\lambda}(c_1,c_2)$ with the Hilbert Property.
\end{remark}

\subsection{Kummer surfaces}

The following proposition is another application of Theorem \ref{Thm:teoremone}.

\begin{proposition}\label{Cor:kummer}
	Let $E_1$ and $E_2$ be two elliptic curves defined over a number field $K$, with positive Mordell-Weil rank. The Kummer surface $S \defeq E_1\times E_2 /\{\pm 1\}$ has the Hilbert Property.
\end{proposition}

\begin{proof}
	A desingularization of $S$, which we denote by $\tilde{S}$, may be obtained as the quotient by $\{\pm 1\}$ of the blow up $\widehat{E_1\times E_2}$ of $E_1 \times E_2$ in the $16$  $2$-torsion points. We denote the set of the images of these points in ${S}$ with $\mathcal{T}$, and the corresponding exceptional lines in $\tilde{S}$ with $\mathcal{L}$. Moreover, we denote by $b: \tilde{S} \rightarrow S$ the just described desingularization morphism, and by $q: E_1 \times E_2 \rightarrow S$ the quotient map.
	
	The surface $\tilde{S}$ has at least three elliptic fibrations, defined over $K$. Two of these, which we denote by $\pi_i, \ i=1,2$ are the following compositions:
	\[
	\tilde{S} \rightarrow S = E_1\times E_2 /\{\pm 1\} \rightarrow E_i/\{\pm 1\}\cong \P_1, \ i =1,2.
	\]
	If $E_1\defeq \{y_1^2z_1=f_1(x_1,z_1)\}$ and $E_2\defeq \{y_2^2z_2=f_2(x_2,z_2)\}$, then, as noted in Remark \ref{Rmk:casoparticolarekummer}, $S$ is birational to the surface defined by the following equation for $([x:y:z],[w_1:x_2]) \in \P_2 \times \P_1$:
	\[
	w_1^2f_2(x,z)=w_2^2f_1(y,z).
	\]
	We have that:
	\begin{equation}\label{Eq:pi3}
	[w_1:w_2] \circ q = [y_1z_2:y_2z_1].
	\end{equation}
	We then define the third fibration, $\pi_3$, to be the extension (as a rational map) to $\tilde{S}$ of the map $([x:y:z],[w_1:w_2])\rightarrow [w_1:w_2]$. Let us check that $\Dom (\pi_3)=\tilde{S}$. The map $[y_1z_2:y_2z_1]$ is well-defined on $\widehat{E_1\times E_2}$. Moreover, since $[y_1z_2:y_2z_1]:\widehat{E_1\times E_2} \rightarrow \P_1$ is invariant by the action of $\{\pm 1 \}$, it induces indeed a well-defined morphism on the quotient $\tilde{S}=\widehat{E_1\times E_2}/\{\pm 1 \}$.
	
	We have that, since $y_i$ is a local parameter at points of order $2$ in $E_i$, and $z_i$ is a local parameter at $O \in E_i$, the morphism $[y_1z_2:y_2z_1]:\widehat{E_1\times E_2} \rightarrow \P_1$ is non-constant on the exceptional lines lying over the points $(T_1,T_2)$, when both $T_1$ and $T_2$ have order $2$ or both have order $0$, and it is constant on the other exceptional lines.
	
	It follows that $\Fix(\pi_1,\pi_2,\pi_3)$ is the union of the $6$ exceptional lines in $\tilde{S}$ lying over the points $(T_1,T_2)$, when exactly one of $T_1$, $T_2$ has order $2$ and the other is $O$. We denote the union of these $6$ points in ${S}$ with $\mathcal{T}^b$, and the corresponding lines in $\tilde{S}$ with $\mathcal{L}^b\defeq b^{-1}(\mathcal{T}^b)$. Since, by hypothesis, $S(K)\supset q(E_1(K)\times E_2(K))$ is Zariski-dense in $S$, the proposition follows from Theorem \ref{Thm:teoremone} and the following lemma.
\end{proof}

\begin{remark}\label{Rmk:zariskidensity}
	In Proposition \ref{Cor:kummer} the hypothesis that the two curves $E_1, E_2$ have positive Mordell-Weil rank is just used to guarantee that $K$-rational points are Zariski-dense in $S$. However, one can remove this hypothesis in the case that the $j$-invariants of $E_1$ and $E_2$ are not both equal to $0$ or $1728$: in fact, under these assumptions, Kuwata and Wang showed in \cite{kummer} that $K$-rational points\footnote{They work in the specific case $K=\Q$, but the part of their paper where they prove Zariski-density of rational points can be rephrased ad litteram over any number field.} are always Zariski-dense.
\end{remark}

\begin{lemma}\label{Prop:simplyconnected}
	The surface $(\tilde{S}\setminus \mathcal{L}^b)/\C$ is (topologically) simply connected.
\end{lemma}

\begin{proof}
	Let $\Lambda_1=<e_1,e_2>$ and  $\Lambda_2=<e_3,e_4>$ be lattices in $\C$ such that $E_1 \cong \C/\Lambda_1$, and $E_2 \cong \C/\Lambda_2$ as analytic spaces.
	We observe that the universal cover of $\tilde{S}\setminus \mathcal{L}=S\setminus \mathcal{T}$ is the following composition \[\C^2 \setminus \frac{1}{2}(\Lambda_1\times \Lambda_2) \rightarrow \C/\Lambda_1 \times \C/\Lambda_2 \setminus q^{-1}(\mathcal{T}) \cong E_1 \times E_2 \setminus q^{-1}(\mathcal{T}) \xrightarrow{q} S\setminus \mathcal{T}. \] Therefore
	\[
	\pi_1({S}\setminus \mathcal{T},p_0)\cong (\Lambda_1\times \Lambda_2) \rtimes <\iota>,
	\]
	where $p_0$ denotes a point infinitesimally near to $q((O,O))\in S$, the action of $\iota$ on $\Lambda_1\times \Lambda_2$ is given by $(a,b)\rightarrow (-a,-b)$, and the element $\iota$ corresponds to a (single) loop around $q((O,O))\in S$ (and hence $\iota^2=1$). For any $2$-torsion point $T$ in $E_1\times E_2$, let $\iota_T \in \pi_1({S}\setminus \mathcal{T},p_0)$ denote the element corresponding to a (single) loop around the point $q(T) \in S$\footnote{This element is well-defined only after a choice of a path between $q(p_0)$ and a point infinitesimally near to $q(T)$ has been made. This choice can be done arbitrarily and it is in fact irrelevant for our purposes. We will assume anyway that the path chosen is the geodetic, using the distance induced by the universal cover \[\R^4 \xrightarrow{(e_1|e_2|e_3|e_4)} \C \times \C \rightarrow \C/\Lambda_1 \times \C/\Lambda_2 \cong E_1 \times E_2 \xrightarrow{q} S,\] where the $\R^4$ on the left is endowed with the Euclidean metric.}. We have that, if $T=\sum_{i=1}^{4}\frac{\epsilon_i}{2}e_i$, where $\epsilon_i \in \{0,1\}$, then
	\[
	\iota_T=\left(\sum_{i=1}^{4}{\epsilon_i}e_i\right)\iota.
	\]
	Let $H \subset (\Lambda_1\times \Lambda_2) \rtimes \Z/2\Z$ denote the minimal normal subgroup containing the $\iota_T$'s, for $T \in \mathcal{T}^g \defeq \mathcal{T} \setminus \mathcal{T}^b$. 
	We note that $(e_1+e_4)/2$, $(e_1+e_3)/2$, $(e_1+e_3+e_4)/2 \in \mathcal{T}^g$, and hence $e_1=(e_1+e_4)+(e_1+e_3)-(e_1+e_3+e_4) \in H$. Analogously one has that $e_i \in H$ for every $1 \leq i \leq 4$. Therefore, $H=(\Lambda_1\times \Lambda_2) \rtimes <\iota>$.
	
	We now observe that, in the blown-up surface $\tilde{S}\setminus \mathcal{L}^b$, the loop $\iota_T$ becomes trivial for any point $T \in  \mathcal{T}^g$. In fact, a small topological neighborhood $L_T^{\epsilon}$ of the exceptional line $L_T\defeq b^{-1}(T) \subset \tilde{S}$ is retractible on $L_T$ itself, which is simply connected. 
	
	Therefore, by Van Kampen's Theorem applied to $\tilde{S} \setminus \mathcal{L}^b= (\tilde{S} \setminus \mathcal{L}) \cup_{T \in \mathcal{T}^g} L_T^{\epsilon}$, we have that
	\[
	\pi_1(\tilde{S} \setminus \mathcal{L}^b)\cong \faktor{(\Lambda_1\times \Lambda_2) \rtimes <\iota>}{H} \cong \{1\},
	\]
	as we wanted to prove.
\end{proof}

\begin{remark}\label{Rmk:Kummerotherproof?}
	Corvaja and Zannier proved in \cite{articoloHP} that there are Zariski-dense $K$-rational points in $E_1 \times E_1 / \{\pm 1\}$ that are not image of $K$-rational points in $E_1 \times E_2$.  It is possible that combining the technique presented in \cite{articoloHP} with other results, for instance the ones contained in \cite{zannierHP}, one may obtain a different proof of Proposition \ref{Cor:kummer}, at least in the case where $E_1=E_2$.
\end{remark}

\ack

The author would like to heartily thank his advisor Umberto Zannier, for having given him the opportunity to work on these topics, and providing important insights. The author would also like to thank Pietro Corvaja for fruitful discussions and for Remark \ref{Rmk:zariskidensity}, as well as the anonymous referees for their comments and remarks. Finally, the author would like to thank Jean-Louis Colliot-Thélène for making him notice that the main result of Swinnerton-Dyer \cite{cubiche-swinndyer} in fact implies the $2$-dimensional case of Theorem \ref{cubiche}.

\bibliographystyle{apalike}      
\bibliography{NonrationalHP}

\end{document}